\documentclass[a4paper,12pt,twoside]{article}

\usepackage{amsmath,amsthm,amscd,amsfonts,amssymb,amsxtra}
\usepackage[hmargin=1.2in,vmargin=1.2in]{geometry}
\usepackage[utf8]{inputenc}
\usepackage{graphicx}
\usepackage{microtype}
\usepackage[auth-sc-lg,affil-it]{authblk}
\usepackage[pdftex,bookmarks,colorlinks=false]{hyperref}
\usepackage{booktabs}
\usepackage{enumitem}
\usepackage{caption,subcaption}
\usepackage[mathscr]{euscript}

\newcommand{\E}{\mu}
\newcommand{\C}{\mathcal{C}}
\newcommand{\V}{\sigma^2}
\newcommand{\iE}{\E_{\varphi}}
\newcommand{\iV}{\V_{\varphi}}
\newcommand{\iiE}{\E_{\varphi^+}}
\newcommand{\iiV}{\V_{\varphi^+}}

\newtheorem{thm}{Theorem}[section]

\theoremstyle{defn}
\newtheorem{defn}{Definition}[section]

\newtheorem{rem}{Remark}[section]

\def\ni{\noindent}

\title{\sc Some Results on the $b$-Colouring Parameters of Graphs}

\author{N. K. Sudev}
\affil{\small Department of Mathematics\\ Vidya Academy of Science \& Technology \\ Thrissur - 680501, Kerala, India.\\ {\tt sudevnk@gmail.com,}}

\author{K. P. Chithra}
\affil{\small Naduvath Mana, Nandikkara\\ Thrissur-680301, Kerala, India.\\ {\tt chithrasudev@gmail.com}}

\author{Johan Kok}
\affil{\small Tshwane Metropolitan Police Department\\ City of Tshwane, Republic of South Africa \\ {\tt kokkiek2@tshwane.gov.za}}

\date{}

\begin{document}
\maketitle

\begin{abstract}
A vertex colouring of a given graph $G$ can be considered as a random experiment. A discrete random variable $X$, corresponding to this random experiment, can be defined as the colour of a randomly chosen vertex of $G$ and a probability mass function for this random variable can be defined accordingly. In this paper, we study the concepts of mean and variance corresponding to the $b$-colouring of $G$ and hence determine the values of these parameters for a number of standard graphs.
\end{abstract}

\ni {\bf Key Words}: Graph Colouring; colouring sum of graphs; colouring mean; colouring variance; $b$-chromatic mean; $b$-chromatic variance; $b^+$-chromatic mean; $b^+$-chromatic variance.  

\vspace{0.2cm}

\ni {\small \bf Mathematics Subject Classification}: 05C15, 62A01.

\section{Introduction}

For all  terms and definitions, not defined specifically in this paper, we refer to \cite{BM1,BLS,CZ, FH,EWW,DBW} and for the terminology of graph colouring, we refer to \cite{CZ1,JT1,MK1}.  For the concepts in Statistics, please see \cite{VS1,SMR1}. Unless mentioned otherwise, all graphs considered in this paper are simple, finite, connected and non-trivial.

A graph colouring is considered to be an assignment of colours, labels or weights to the elements of the graphs concerned. Unless mentioned otherwise, by graph colouring we mean the vertex colouring of the graph under consideration. For several decades, graph colouring problems have become an extremely useful models for many theoretical and practical problems.

A \textit{proper colouring} of a graph $G$, is a vertex colouring of it in such a way that no two adjacent vertices in $G$ have the same colour. The \textit{chromatic number} of a graph $G$, denoted by $\chi(G)$,  is the minimum number of colours required in a proper colouring of $G$. A proper $k$-colouring , denoted by $\C$ of a graph $G$ can be usually written as $\C = \{c_1,c_2, c_3, \ldots,c_k\}$ or equivalently $\C = \{1,2,3, \ldots,k\}$. A colour class with respect to a colour $c_i$ of $\C$ is the set of all vertices in $G$ having the colour $c_i$ and is denoted by $\C_i$. The \textit{strength} of a colour $c_i$ of in $\C$ of $G$ is the cardinality of its colour class $\C_i$ and is denoted by $\theta(c_i)$.

Note that a vertex colouring of a graph $G$ can be considered as a random experiment. Let $\C = \{c_1,c_2, c_3, \ldots,c_k\}$ be a proper $k$-colouring of $G$ and let $X$ be a discrete random variable (\textit{d.r.v}) which denotes the colour $c_i$ (or the subscript $i$ of the colour $c_i$) of a randomly selected vertex of $G$. Therefore, since the sum of all weights of colours of $G$ is the order of $G$, the real valued function $f(i)$ defined by 
\begin{equation}\label{eqn-1}
f(i)= 
\begin{cases}
\frac{\theta(c_i)}{|V(G)|}; &  x=1,2,3,\ldots,k\\
0; & \text{elsewhere}.
\end{cases}
\end{equation} 
is the \textit{probability mass function} (\textit{p.m.f}) of the \textit{d. r. v.} $X$. If the context is clear, we can say that $f(i)$ is the \textit{p.m.f} of the corresponding graph $G$ with respect to the given colouring $\C$.

Using the above terminology, the concepts of mean and variance of a random variable can be extended to the colouring mean and colouring variance of given graphs with respect to different types of colourings of graphs. In this paper, we study the colouring mean and variance of corresponding to the $b$-colouring of certain fundamental graph classes.

\section{$b$-Chromatic Mean and Variance of Graphs}

A \textit{$b$-colouring} of a graph $G$ is a colouring of the vertices of $G$ such that each colour class contains at least one vertex that has a neighbour in all other colour classes. The \textit{$b$-chromatic number} of a graph $G$, denoted by $\varphi(G)$, is the largest integer $k$ such that $G$ admits a $b$-colouring with $k$ colours (see \cite{IM1}).  

An important and relevant result on the bounds of $b$-chromatic number of a given graph $G$ is \begin{equation}\label{eqn-2}
\chi(G)\le\varphi(G)\le\Delta(G)+1. 
\end{equation}
where $\Delta(G)$ is the maximum degree of (the vertices) in $G$ (see \cite{KM1}).

Several interesting studies on the $b$-chromatic numbers and related parameters have been done (see \cite{EK1,KS1,KSC1,KM1,ES1,LS1,SSCK1,VI1,VI2,VV1}). Motivated by these studies, we extend the concepts of mean and variance of random variables to the $b$-colouring concepts of graphs as follows.

\begin{defn}{\rm
Let $\C = \{c_1,c_2, c_3, \ldots,c_k\}$ be a $b$-colouring of a given graph $G$ and $\theta(c_i)$ denotes the strength of a particular colour $c_i$ in the colouring $\C$ of $G$. Then, the \textit{$b$-colouring mean} of a given graph $G$ with respect to the colouring $\C$ is denoted by $\E_{\C}(X)$ and is defined to be 
\begin{equation}\label{eqn-3}
\E_{\C}(X) =\sum\limits_{i=1}^{n}i\,f(i)
\end{equation}
}\end{defn}

\begin{defn}{\rm
Let $\C = \{c_1,c_2, c_3, \ldots,c_k\}$ be a $b$-colouring of a given graph $G$ and $\theta(c_i)$ denotes the strength of a particular colour $c_i$ in the colouring $\C$ of $G$. The \textit{$b$-colouring variance} of a colouring $\C$ of a given graph $G$, denoted by $\V_{\C}(G)$, is defined as 
\begin{equation}\label{eqn-4}
\V_{\C}(X) =\sum\limits_{i=1}^{n}i^2\,f(i)-\left(\sum\limits_{i=1}^{n}i\,f(i)\right)^2 
\end{equation}	
}\end{defn}

It can be noted that we can have several $b$-colourings of a given graph $G$ using the same colours in the proper colouring $\C$. They differ only in the cardinality of different colour classes. It naturally raises some questions about the possible lower and upper bounds of the set of all such possible $b$-colouring means of the graph $G$ under consideration. Hence, we have the following definitions.

\begin{defn}{\rm 
A $b$-colouring mean of a graph $G$, with respect to a $b$-colouring $\C$ is said to be a {\em $b$-chromatic mean} or \textit{$\varphi$-chromatic mean} of $G$, if $\C$ is a $b$-colouring of $G$ such that the $b$-colouring mean is minimum. The $b$-chromatic mean of a graph $G$ is denoted by $\iE(G)$. }
\end{defn}

\begin{defn}{\rm 
The \textit{$b$-chromatic variance} of $G$, denoted by $\iV(G)$, is a $b$-colouring variance of $G$ with respect to a $b$-colouring $\C$ of $G$ which yields the $b$-chromatic mean.  }
\end{defn}

\begin{defn}{\rm 
A $b$-colouring mean of a graph $G$, with respect to a $b$-colouring $\C$ is said to be a {\em $b^+$-chromatic mean} or \textit{$\varphi^+$-chromatic mean} of $G$, if $\C$ is a $b$-colouring of $G$ such that the $b$-colouring mean is maximum. The $b^+$-chromatic mean of a graph $G$ is denoted by $\iiE(G)$. }
\end{defn}

\begin{defn}{\rm 
The \textit{$b^+$-chromatic variance} of $G$, denoted by $\iiV(G)$, is a $b$-colouring variance of $G$ with respect to a $b$-colouring $\C$ of $G$ which yields the $b^+$-chromatic mean.  }
\end{defn}

Note that the $\iE(G)$ is the minimum $b$-colouring mean and $\iiE(G)$ is the maximum $b$-colouring mean when we consider all possible $b$-colourings of a given graph $G$, using the same colours.

\begin{rem}{\rm 
Also note that the $b^+$-colouring parameters of a graph $G$ can be found out by reversing the colouring pattern of $G$ using the same $b$-colouring which provides $b$-chromatic parameters of $G$. }
\end{rem}

In the following section, we study these particular types of $b$-colouring means and variances of different fundamental graphs such as paths, cycles, complete graphs, bipartite graphs etc. 

\begin{rem}	{\rm 
As a proper colouring, any $b$-colouring of a complete graph $K_n$ must contain $n$ distinct colours such that every colour class is a singleton set. Hence, the corresponding \textit{p.m.f.} is 
$$f(i)=
\begin{cases}
\frac{1}{n}; & i=1,2,3,\ldots,n,\\
0; & \text{elsewhere}.
\end{cases}$$ . Hence, we can see that $b$-colouring of a complete graph follows discrete uniform distribution $DU(n)$. This $b$-colouring is also a $b^+$-colouring of $K_n$ as well. Therefore, $\iE(K_n)=\iiE(K_n)=\frac{n+1}{2}$ and $\iV(K_n)=\iiV(K_n)=\frac{n^2-1}{12}$.	}
\end{rem}

In the following theorem, we determine the \textit{p.m.f.} and hence the $\varphi$-chromatic mean and variance of a path $P_n$ on $n\ge 4$ vertices.

\begin{thm}\label{Thm-1}
The $b$-chromatic mean of a path $P_n$ is given by
\begin{equation*}
\iE(P_n)=
\begin{cases}
\frac{3}{2}; & \text{if $n=2$}\\
\frac{4}{3}; & \text{if $n=3$}\\
\frac{3n+2}{2n}; & \text{if $n\ge 4$ is even}\\
\frac{3n+3}{2n}; & \text{if $n>4$ is odd}
\end{cases}
\end{equation*}
and the $b$-chromatic variance of $P_n$ is 
\begin{equation*}
\iV(P_n)=
\begin{cases}
\frac{1}{4}; & \text{if $n=2$}\\
\frac{2}{9}; & \text{if $n=3$}\\
\frac{n^2+8n-4}{4n^2}; & \text{if $n\ge 4$ is even}\\
\frac{n^2+8n-9}{4n^2}; & \text{if $n>4$ is odd}
\end{cases}
\end{equation*}
\end{thm}
\begin{proof}
First note that the $b$ chromatic number of $P_n$ is $2$ if $n=2,3$ and is $3$ if $n\ge 4$. For $P_2$, both colour classes are singleton sets and hence the \textit{p.m.f.} $f(i)$ is given by $f(i)=\frac{1}{2}$ for $i=1,2$ and $0$ elsewhere. Therefore, by Equation \eqref{eqn-3}, $\iE(P_2)=\frac{3}{2}$ and by Equation \eqref{eqn-4}, $\iV(P_2)=\frac{1}{4}$. 

For $P_3$, two vertices have colour $c_1$ and one vertex has colour $c_2$ and hence the corresponding \textit{p.m.f.} is given by 
$$f(i)=
\begin{cases}
\frac{2}{3}; & i=1,\\
\frac{1}{3}; & i=2\\
0 & \text{elsewhere}.
\end{cases}$$
Therefore, $\iE(P_3)=\frac{4}{3}$ and $\iV(P_3)=\frac{2}{9}$.

For $n\ge 4$, we know that any $b$-colouring of $P_n$ contains $3$ colours, say $c_1,c_2,c_3$, in which we can restrict the occurrence of the colour $c_3$ to a single vertex and the colour $c_1$ to maximum number of remaining vertices. Then, we have to consider the following cases.

\textit{Case-1:} Let $n$ be an even integer. Then, we can consider a $b$-colouring such that the colour class of $c_1$ has $\frac{n}{2}$ vertices, the colour class of $c_2$ has $\frac{n-2}{2}$ vertices and that of $c_3$ is a singleton set. Therefore, the corresponding \textit{p.m.f.} is  
$$f(i)=
\begin{cases}
\frac{1}{2}; & i=1,\\
\frac{n-2}{2n}; & i=2\\
\frac{1}{n}; & i=3\\
0 & \text{elsewhere}.
\end{cases}$$

Therefore, $\iE(P_n)=\frac{3n+2}{2}$ and $\iV(P_3)=\frac{n^2+8n-4}{4n^2}$.

\textit{Case-2:} Let $n$ be an odd integer. Then, we have a $b$-colouring such that the colour classes of $c_1$ and $c_2$ have $\frac{n-1}{2}$ vertices and that of $c_3$ is a singleton set. Therefore, the corresponding \textit{p.m.f.} is  

$$f(i)=
\begin{cases}
\frac{n-1}{2}; & i=1,2\\
\frac{1}{n}; & i=3\\
0 & \text{elsewhere}.
\end{cases}$$

Therefore, $\iE(P_n)=\frac{3n+3}{2n}$ and $\iV(P_n)=\frac{n^2+8n-9}{4n^2}$.
\end{proof}

In a similar way, we can find the values of these parameters with respect to the $b$-colouring of cycles $C_n$ as follows.

\begin{thm}\label{Thm-2}
	The $b$-chromatic mean of a cycle  $C_n$ is given by
	\begin{equation*}
		\iE(C_n)=
		\begin{cases}
			\frac{3}{2}; & \text{if $n=4$}\\
			\frac{3n+6}{2n}; & \text{if $n$ is even, $n\ne 4$},\\
			\frac{3n+3}{2n}; & \text{if $n$ is odd}
		\end{cases}
	\end{equation*}
	and the $b$-chromatic variance of $C_n$ is 
	\begin{equation*}
		\iV(C_n)=
		\begin{cases}
			\frac{1}{4}; & \text{if $n=4$}\\
			\frac{n^2+16n+36}{4n^2}; & \text{if $n$ is even and $n\ne 4$}\\
			\frac{n^2+8n-9}{4n^2}; & \text{if $n$ is odd}.
		\end{cases}
	\end{equation*}
\end{thm}
\begin{proof}
First note that the $b$ chromatic number of $C_4$ is $2$ and that of $C_n$ is $3$ for $n\ne 4$. For $C_4$, two vertices each have colours $c_1$ and $c_2$ respectively. Then, the corresponding \textit{p.m.f} is $f(i)=\frac{1}{2}$ for $i=1,2$ and $0$ elsewhere. Hence by \eqref{eqn-3} and \eqref{eqn-4}, we have $\iE(C_4)=\frac{3}{2}$ and $\iV(C_4)=\frac{1}{4}$. 

If $n\ne 4$, as mentioned in the proof previous theorem, we can find a $b$-colouring in which the colour $c_3$ is assigned to exactly one or two vertices of $C_n$ as per requirements. Here, we have to consider the following cases.

\textit{Case-1:} If $n$ is even, we have to assign colour $c_3$ to two vertices of $C_n$ and the colours $c_1$ and $c_2$ are assigned to the remaining vertices equally. Therefore, we have the corresponding \textit{p.m.f.} is given by 
$$f(i)=
\begin{cases}
\frac{n-2}{2n}; & i=1,2\\
\frac{2}{n}; & i=3\\
0 & \text{elsewhere}.
\end{cases}$$
Therefore, By \eqref{eqn-3} and \eqref{eqn-4}, we have, $\iE(C_n)=\frac{3n+6}{2n}$ and $\iV(C_n)=\frac{n^2+16n+36}{4n^2}$.

\textit{Case-1:} If $n$ is odd, the colour $c_3$ is to be assigned to exactly one vertex of $C_n$ and the colours $c_1$ and $c_2$ are assigned to $\frac{n-1}{2}$ vertices each. Therefore, we have the corresponding \textit{p.m.f.} is given by 
$$f(i)=
\begin{cases}
\frac{n-1}{2n}; & i=1,2\\
\frac{1}{n}; & i=3\\
0 & \text{elsewhere}.
\end{cases}$$
Therefore, By \eqref{eqn-3} and \eqref{eqn-4}, we have, $\iE(C_n)=\frac{3n+3}{2n}$ and $\iV(C_n)=\frac{n^2+8n-9}{4n^2}$.
\end{proof}

Next, the $b$-chromatic mean and variance of wheel graphs, defined by $W_{n+1}=C_n+K_1$, are determined in the following theorem.

\begin{thm}\label{Thm-3}
The $b$-chromatic mean of a wheel graph $W_{n+1}$ is given by
\begin{equation*}
\iE(W_{n+1})=
\begin{cases}
\frac{9}{5}; & \text{if $n=4$},\\
\frac{3n+14}{2n+2}; & \text{if $n$ is even, $n\ne 4$},\\
\frac{3n+11}{2n+2}; & \text{if $n$ is odd}
\end{cases}
\end{equation*}
and the $b$-chromatic variance of $C_n$ is 
\begin{equation*}
\iV(W_{n+1})=
\begin{cases}
\frac{14}{25}; & \text{if $n=4$},\\
\frac{n^2+42n-80}{4(n+1)^2}; & \text{if $n$ is even and $n\ne 4$},\\
\frac{n^2+34n-31}{4(n+1)^2}; & \text{if $n$ is odd}.
\end{cases}
\end{equation*}
\end{thm} 
\begin{proof}
The $b$ chromatic number of $W_5$ is $3$ and that of $W_{n+1}$ is $4$ for $n\ne 4$. For $W_5$, two vertices each have colours $c_1$ and $c_2$ respectively and the central vertex has the colour $c_3$. Then, the corresponding \textit{p.m.f} is 
$$f(i)=
\begin{cases}
\frac{2}{5}; & i=1,2,\\
\frac{1}{5}; & i=3,\\
0; \text{elsewhere}.
\end{cases}$$. 
Hence, by \eqref{eqn-3} and \eqref{eqn-4}, we have $\iE(W_5)=\frac{9}{5}$ and $\iV(W_5)=\frac{14}{25}$.

\vspace{0.25cm}

Next, let $n\ne 4$. Then, the $b$- chromatic number of $W_{n+1}$ is $4$. If $n$ is even, we can find a $b$-colouring such that $\frac{n-2}{2}$ vertices have colours $c_1$ and $c_2$ each and two vertices have colour $c_3$ and one vertex has colour $c_4$. Hence, the corresponding \textit{p.m.f} is given by 
\begin{equation*}
f(i)=
\begin{cases}
\frac{n-2}{2(n+1)}; & i=1,2,\\
\frac{2}{n+1}; & i=3,\\
\frac{1}{n+1}; & i=4,\\
0; & \text{elsewhere}.
\end{cases}
\end{equation*} 
Hence, we have $\iE(W_{n+1})=\frac{3n+14}{2n+2}$ and $\iV(W_{n+1})=\frac{n^2+42n-80}{4(n+1)^2}$.

\vspace{0.25cm}

If $n$ is odd, we can find a $b$-colouring such that $\frac{n-1}{2}$ vertices have colours $c_1$ and $c_2$ each and one vertex each has colour $c_3$ and $c_4$. Hence, the corresponding \textit{p.m.f} is given by 
\begin{equation*}
f(i)=
\begin{cases}
\frac{n-1}{2(n+1)}; & i=1,2,\\
\frac{1}{n+1}; & i=3,4\\
0; & \text{elsewhere}.
\end{cases}
\end{equation*} 
Hence, we have $\iE(W_{n+1})=\frac{3n+11}{2n+2}$ and $\iV(W_{n+1})=\frac{n^2+34n-31}{4(n+1)^2}$.
\end{proof}

Another cycle related graph that catches attention in this context is a sunlet graph $S_n$ which is defined by $S_n=C_n\odot K_1$, where $\odot$ represent the corona product of two graphs (see \cite{FH} for the definition of corona of two graphs). In the following theorem, we determine the $b$-chromatic mean and variance of sunlet graphs.

\begin{thm}\label{Thm-4}
The $b$-chromatic mean of a sunlet graph is 
$$\iE(S_n)=
\begin{cases}
\frac{5}{3}; & n=3,\\
\frac{5}{2}; & n=4,\\
\frac{17}{10}; & n=5,\\
\frac{3n+7}{2n}; & n\ge 6
\end{cases}$$
and the $b$-chromatic variance of the sunlet graph $S_n$ is 
$$\iV(S_n)=
\begin{cases}
\frac{5}{9}; & n=3,\\
\frac{5}{4}; & n=4,\\
\frac{61}{100}; & n=5,\\
\frac{n^2+35n-49}{4n^2}; & n\ge 6.
\end{cases}$$
\end{thm}
\begin{proof}
For $n=3$, a $b$-colouring of $S_n$ has $3$ colours. We can find a $b$-colouring such that $3$ vertices have colour $c_1$, $2$ vertices have colour $c_2$ and $1$ vertex has the colour $c_3$. Then, the corresponding \textit{p.m.f.} is given by
$$f(i)=
\begin{cases}
\frac{1}{2}; & i=1,\\
\frac{1}{3}; & i=2,\\
\frac{1}{6}; & i=3,\\
0; & \text{elsewhere}.
\end{cases}$$
Hence, we have $\iE(S_3)=\frac{5}{3}$ and $\iV(S_3)=\frac{5}{9}$.

\vspace{0.25cm}

A $b$-colouring of $S_4$ has $4$ colours. We can find a $b$-colouring such that $2$ vertices have colours $c_1,c_2,c_3$ and $c_4$. Then, the corresponding \textit{p.m.f.} is given by
$$f(i)=
\begin{cases}
\frac{1}{4}; & i=1,2,3,4\\
0; & \text{elsewhere}.
\end{cases}$$
Hence, we have $\iE(S_4)=\frac{5}{2}$ and $\iV(S_4)=\frac{5}{4}$.

\vspace{0.25cm}

A $b$-colouring of $S_5$ has only $3$ colours. We can find a $b$-colouring such that $5$ vertices have colour $c_1$, $3$ vertices have colour $c_2$ and $2$ vertices have colour $c_3$. Then, the corresponding \textit{p.m.f.} is given by
$$f(i)=
\begin{cases}
\frac{1}{2}; & i=1,\\
\frac{3}{10}; & i=2,\\
\frac{1}{5}; & i=3,\\
0; & \text{elsewhere}.
\end{cases}$$
Hence, we have $\iE(S_5)=\frac{17}{10}$ and $\iV(S_5)=\frac{61}{100}$.

\vspace{0.25cm}

If $n\ge 6$, then we can find a $b$-colouring of $S_n$ containing $4$ colours in such a way that $n-1$ vertices have colour $c_1$, $n-3$ vertices have colour $c_2$ and two vertices each have colours $c_3$ and $c_4$. Then, the corresponding \textit{p.m.f.} is 
$$f(i)=
\begin{cases}
\frac{n-1}{2n}; & i=1,\\
\frac{n-3}{2n}; & i=2,\\
\frac{1}{n}; & i=3,4,\\
0; & \text{elsewhere}.
\end{cases}$$
Therefore, we have $\iE(S_n)=\frac{3n+7}{2n}$ and $\iV(S_n)=\frac{n^2+35n-49}{4n^2}$.
\end{proof}

Another cycle related graph, we consider in  this context is a closed ladder $CL_n$, which is a graph obtained by joining every pair of the corresponding vertices of two cycles by an edge. That is, $CL_n=\C_n\Box P_2$. The following theorem discusses the $b$-chromatic mean and variance of a closed ladder graph.

\begin{thm}\label{Thm-5}
The $b$-chromatic mean of a closed ladder $CL_n$ graph is 
	$$\iE(CL_n)=
	\begin{cases}
	5; & n=3,\\
	\frac{5}{2}; & n=4,\\
	\frac{23}{10}; & n=5,\\
	\frac{23}{12}; & n=6,\\
	\frac{3n+8}{2n}; & n\ge 7, \text{$n$ is odd}\\
	\frac{3n+7}{2n}; & n\ge 7, \text{$n$ is even}.
	\end{cases}$$
	and the $b$-chromatic variance of graph $CL_n$ is 
	$$\iV(CL_n)=
	\begin{cases}
	2; & n=3,\\
	\frac{5}{4}; & n=4,\\
	\frac{131}{144}; & n=5,\\
	\frac{n^2+28n-64}{4n^2}; & n\ge 7, \text{$n$ is odd}\\
	\frac{n^2+24n-49}{4n^2}; & n\ge 7, \text{$n$ is even}.
	\end{cases}$$
\end{thm}
\begin{proof}
If $n=3$, we can find a $b$-colouring of $CL_3$ containing $3$ colours such that the corresponding \textit{p.m.f.} is
$$f(i)=
\begin{cases}
\frac{1}{3}; & i=1,2,3,\\
0; & \text{elsewhere}.
\end{cases}$$
Hence, we have $\iE(CL_3)=2$ and $\iV(CL_3)=2$. 

If $n=4$, we can find a $b$-colouring of $CL_4$ containing $4$ colours such that the corresponding \textit{p.m.f.} is
$$f(i)=
\begin{cases}
\frac{1}{4}; & i=1,2,3,4\\
0; & \text{elsewhere}.
\end{cases}$$
Hence, we have $\iE(CL_4)=\frac{5}{2}$ and $\iV(CL_n)=\frac{5}{4}$.

We can find a $b$-colouring of $CL_5$ containing $4$ colours such that the corresponding \textit{p.m.f.} is
$$f(i)=
\begin{cases}
\frac{3}{10}; & i=1,2,\\
\frac{1}{5}; & i=3,4\\
0; & \text{elsewhere}.
\end{cases}$$
Hence, we have $\iE(CL_5)=\frac{23}{10}$ and $\iV(CL_5)=\frac{121}{100}$.
 
We can find a $b$-colouring of $CL_6$ containing $4$ colours such that the corresponding \textit{p.m.f.} is
$$f(i)=
\begin{cases}
\frac{5}{12}; & i=1,\\
\frac{1}{3}; & i=2,\\
\frac{1}{6}; & i=3,\\
\frac{1}{12}; & i=4,\\
0; & \text{elsewhere}.
\end{cases}$$
Hence, we have $\iE(CL_6)=\frac{23}{12}$ and $\iV(CL_6)=\frac{131}{144}$.

Let $n\ge 7$. If $n$ is odd, there exist a $b$-colouring for $CL_n$ such that the corresponding \textit{p.m.f.} is given by  
$$f(i)=
\begin{cases}
\frac{n-2}{2n}; & i=1,\\
\frac{n-3}{2n}; & i=2,\\
\frac{2}{n}; & i=3,\\
\frac{1}{2n2}; & i=4,\\
0; & \text{elsewhere}.
\end{cases}$$
Therefore, we have $\iE(CL_n)=\frac{3n+8}{2n}$ and $\iV(CL_n)=\frac{n^2+28n-64}{4n^2}$.

If $n$ is even, there exist a $b$-colouring for $CL_n$ such that the corresponding \textit{p.m.f.} is given by  
$$f(i)=
\begin{cases}
\frac{n-2}{2n}; & i=1,2\\
\frac{3}{2n}; & i=3,\\
\frac{1}{2n2}; & i=4,\\
0; & \text{elsewhere}.
\end{cases}$$
Therefore, we have $\iE(CL_n)=\frac{3n+7}{2n}$ and $\iV(CL_n)=\frac{n^2+24n-49}{4n^2}$.
\end{proof}

\section{Conclusion}

In this paper, we extended two important statistical parameters to the concepts of $b$-colouring of certain fundamental graphs and determined their values for certain graphs and graph classes. More problems in this area are still open. 

Determining the sum, mean and variance corresponding to the $b$-colouring of certain generalised graphs like generalised Petersen graphs, are some of the open problems which seem to be promising for further investigations. Studies on the sum, mean and variance corresponding to different types of edge colourings, map colourings, total colourings etc. of graphs also offer much for future studies.

Different colouring parameters can be used to model important problems in project management, communication networks, optimisation techniques, distribution networks etc. and in several industrial situations. 

We can associate many other parameters to graph colouring and other notions like covering, matching etc.  All these facts highlight a wide scope for future studies in this area.

\section*{Acknowledgement}

The first author of this article would like to dedicate this paper to the memory Prof. (Dr.) D. Balakrishnan, Founder Academic Director, Vidya Academy of Science and Technology, Thrissur, India., who had been his mentor, the philosopher and the role model in teaching and research.

\end{document}